\documentclass[10pt,reqno]{amsart}
\usepackage{amssymb,amsmath,amsthm,mathrsfs}
\usepackage{amscd}
\usepackage{enumerate}
\usepackage{enumitem}
\usepackage{graphicx}
\usepackage{tikz-cd}
\usepackage{color}
\usepackage[dvipsnames]{xcolor}
\usepackage{enumitem}
\usepackage[pagebackref=true, colorlinks=true, citecolor=blue]{hyperref}
\usepackage[misc]{ifsym}
\usepackage{epsfig} %For pictures: screened artwork should be set up with an 85 or 100 line screen
\usepackage{epstopdf} %This is to transfer .eps figure to .pdf figure; please compile your article using PDFLaTex or PDFTeXify.
\usepackage{mathabx}
\usepackage[normalem]{ulem}
\usepackage{lineno}
\usepackage{mathtools}%                  http://www.ctan.org/pkg/mathtools
\usepackage[tableposition=top]{caption}% http://www.ctan.org/pkg/caption
\usepackage{booktabs,dcolumn}%           http://www.ctan.org/pkg/dcolumn + http://www.ctan.org/pkg/booktabs
\usepackage{stmaryrd}
\usepackage{cancel}
\usepackage{ulem}

\usepackage{bm}

\DeclareFontFamily{OT1}{rsfs}{}
\DeclareFontShape{OT1}{rsfs}{n}{it}{<-> rsfs10}{}
\DeclareMathAlphabet{\mathscr}{OT1}{rsfs}{n}{it}

\theoremstyle{plain}

\numberwithin{equation}{section}
\newtheorem{theorem}{Theorem}[section]
\newtheorem{proposition}[theorem]{Proposition}
\newtheorem{lemma}{Lemma}[section]

\theoremstyle{definition}

\newtheorem{definition}[theorem]{Definition}

\def\<{\langle} \def\>{\rangle}

\newcommand{\Bu}{{\boldsymbol{u}}}

\newcommand{\Be}{{\boldsymbol{e}}}

\newcommand{\Bx}{\boldsymbol{x}}
\def\dd{\mathop{}\!\mathrm{d}}

\newcommand{\ignore}[1]{}

\newcommand{\stkout}[1]{\ifmmode\text{\sout{\ensuremath{#1}}}\else\sout{#1}\fi}
\usepackage{algorithm, algorithmic}
 \let\oldequation\equation
 \let\oldendequation\endequation
 \renewenvironment{equation}
 {\linenomathNonumbers\oldequation}
 {\oldendequation\endlinenomath}
 \let\oldalign\align
 \let\oldendalign\endalign
 \renewenvironment{align}
 {\linenomathNonumbers\oldalign}
 {\oldendalign\endlinenomath}

 \let\oldgather \gather
 \let\oldendgather\endgather
 \renewenvironment{gather}
 {\linenomathNonumbers\oldgather}
 {\oldendgather\endlinenomath}

\title[Convection in multilayer porous media]{Vanishing layer thickness limit of convection in multilayer porous media}

\author{Kaijian Sha$^{{\href{mailto:kjsha11@eitech.edu.cn}{\textrm{\Letter}}}1}$}
\author{Xiaoming Wang $^{{\href{mailto:wxm.math@outlook.com}{\textrm{\Letter}}}1,2}$}
\address{$^1$ School of Mathematical Sciences, Eastern Institute of Technology, Ningbo, China}
\address{$^2$ Department of Mathematics and Statistics, Missouri University of Science and Technology, Rolla, MO, USA}


\thanks{$^*$Corresponding author: Xiaoming Wang}

\thanks{This work is supported by NSFC grant 12271237 }

\begin{document}

%\linenumbers

\begin{abstract} 
Within the Darcy-Boussinesq framework for convection in multilayered porous media, we investigate the singular limit in which the thickness of one layer tends to zero. We establish that the solution of the full system converges to that of the corresponding limiting model with one fewer layer. The convergence is established in two complementary senses:  (i) strong $L^{2}$-convergence over arbitrary finite time intervals, and (ii) upper semi-continuity of the global attractors describing the large-time asymptotic behavior.
\end{abstract}
 \maketitle 
%\tableofcontents 

\section{Introduction}
Convection in layered porous media arises in a wide range of natural and engineering contexts, including geophysical and industrial processes \cite{bickle2007modelling, hewitt2022jfm, hewitt2014jfm, hewitt2020jfm, huppert2014ar,  mckibbin1980jfm, mckibbin1981heat, mckibbin1983thermal, sahu2017tansp, salibindla2018jfm, wooding1997convection}. A standard approach is to model each layer as a porous medium governed by Darcy's law with buoyancy, coupled to the advection-diffusion equation for solute or contaminant transport. Interfaces between layers are customarily treated as sharp, giving rise to the so-called sharp interface model. The well-posedness of this model in multilayer settings has recently been established in \cite{CNW2025}, and the interfacial boundary conditions have been rigorously justified via the sharp material interface limit \cite{DW2025}.

In many applications, however, one of the layers may be significantly thinner than the others. From a heuristic perspective, one might then neglect the thin layer and arrive at a reduced model with one fewer layer, while leaving all other aspects unchanged. The main contribution of the present work is to rigorously justify this reduction. Specifically, we prove that the solution of the full multilayer system converges to that of the corresponding limiting model with one fewer layer, both in terms of $L^{2}$-convergence over arbitrary finite time intervals and in terms of the long-time dynamics as characterized by the convergence of global attractors. The key in our analysis is uniform (in the thickness of the thin layer) estimates of our solutions, and uniform in time estimates of the solutions in the long-time behavior case.

The study of effective equations in the presence of thin layers has a long history, beginning with the pioneering work of Keller when he derived the effective conductivity of composite material. A more modern and comprehensive treatment can be found in the review work \cite{AV2018}. What we have treated here is a relatively simple case of vanishing thickness instead of simultaneous vanishing thickness and permeability, but in a nonlinear setting as opposed to linear problems in most literatures. 

The remainder of the paper is organized as follows. Section 2 recalls the general multilayer model. In Section 3 we establish preliminary estimates. Section 4 is devoted to uniform (in thickness) estimates in the presence of a thin layer, together with uniform long-time bounds for the solutions. In Section 5 we demonstrate the convergence to the limiting model.

% The title of section 2: 
\section{The mathematical model}
For convection in layered domain, for example $\Omega = (0,L)\times(-H,0)$, the governing equations are the following system \cite{NB2017},
\begin{equation}\label{Sharp}
\left\{\begin{aligned}
  &\Bu=-\frac{K}{\mu}\left(\nabla P + \rho_{0}(1+\alpha\phi)g\Be_{z}\right),\\
  &\nabla \cdot \Bu =0,\\
  &b\partial_t \phi  + \Bu\cdot\nabla\phi - \nabla\cdot(bD\nabla\phi)=0,
\end{aligned}\right.
\end{equation}
Here $\Bu$, $\phi$, and $P$ are the unknown fluid velocity, concentration, and pressure, respectively; $\rho_{0}$, $\alpha$, $\mu$, $g$ are the constant reference fluid density, constant expansion coefficient, constant dynamic viscosity, and the gravity acceleration constant, respectively; and $\Be_{z}$ stands for the unit vector in the vertical-direction. In addition, $K$, $b$, $D$ represent the permeability, porosity, and diffusivity coefficients respectively. 

For the idealized sharp interface model, the domain $\Omega$ is divided into $l$ layers by $l-1$ interfaces located at $z= z_i \in (-H,0), i = 1,2,\cdots,l-1$. The $i$-th layer $\Omega_i$ is denoted by 
\[\Omega_i: = \{\Bx= (x,z)\in \Omega:~z\in (z_{i-1},z_i)\},~~~~\text{ for }i=1,2,\cdots,l.\]

In each layer $\Omega_i$, the permeability, porosity, and diffusivity coefficients are assumed to be constant. Namely,
\[
K=K(\boldsymbol{x})=K_i, \quad b=b(\boldsymbol{x})=b_i, \quad D=D(\boldsymbol{x})=D_i, \quad \boldsymbol{x}\in\Omega_i, \quad 1\leq i\leq l,
\]
for a set of constants $\{K_i,b_i,D_i\}_{i=1}^{l}$.

\begin{figure}[htbp]
    \centering

\tikzset{every picture/.style={line width=0.75pt}} %set default line width to 0.75pt        

\begin{tikzpicture}[x=0.75pt,y=0.75pt,yscale=-1,xscale=1]

%Straight Lines [id:da6135873122990148] 
\draw    (230,66) -- (400,66) ;
%Straight Lines [id:da641280334322462] 
\draw    (230,238) -- (400,238) ;
%Straight Lines [id:da2643712137709442] 
\draw  [dash pattern={on 0.84pt off 2.51pt}]  (230,66) -- (230,238) ;
%Straight Lines [id:da4832788241319341] 
\draw  [dash pattern={on 0.84pt off 2.51pt}]  (400,66) -- (400,238) ;
%Straight Lines [id:da18196643446705674] 
\draw  [dash pattern={on 3.75pt off 3pt on 7.5pt off 1.5pt}]  (230,151) -- (400,151) ;
%Straight Lines [id:da6233992328530896] 
\draw  [dash pattern={on 3.75pt off 3pt on 7.5pt off 1.5pt}]  (229.5,171) -- (399.5,171) ;
%Straight Lines [id:da37864337206122467] 
\draw  [dash pattern={on 3.75pt off 3pt on 7.5pt off 1.5pt}]  (230,102) -- (400,102) ;
%Shape: Brace [id:dp45428938195597734] 
\draw   (228.8,151.6) .. controls (226.27,151.56) and (224.99,152.8) .. (224.95,155.33) -- (224.95,155.33) .. controls (224.89,158.94) and (223.6,160.72) .. (221.07,160.68) .. controls (223.6,160.72) and (224.83,162.54) .. (224.77,166.15)(224.8,164.53) -- (224.77,166.15) .. controls (224.73,168.68) and (225.97,169.96) .. (228.5,170) ;

% Text Node
\draw (409.33,60.4) node [anchor=north west][inner sep=0.75pt]  [font=\small]  {$z_{0} =0$};
% Text Node
\draw (409.67,228.07) node [anchor=north west][inner sep=0.75pt]  [font=\small]  {$z_{l} =-H$};
% Text Node
\draw (408.33,147) node [anchor=north west][inner sep=0.75pt]  [font=\small]  {$z_{j-1}$};
% Text Node
\draw (410.33,165) node [anchor=north west][inner sep=0.75pt]  [font=\small]  {$z_{j}$};
% Text Node
\draw (410.67,97) node [anchor=north west][inner sep=0.75pt]  [font=\small]  {$z_{1}$};
% Text Node
\draw (308.67,83) node [anchor=north west][inner sep=0.75pt]  [font=\small]  {$\Omega _{1}$};
% Text Node
\draw (311.33,155) node [anchor=north west][inner sep=0.75pt]  [font=\small]  {$\Omega _{j}$};
% Text Node
\draw (204.25,153.57) node [anchor=north west][inner sep=0.75pt]  [font=\small]  {$h_{j}$};
% Text Node
\draw (405.5,192.65) node [anchor=north west][inner sep=0.75pt]  [font=\small]  {$\cdots $};
% Text Node
\draw (405,118.65) node [anchor=north west][inner sep=0.75pt]  [font=\small]  {$\cdots $};
% Text Node
\draw (307.5,118.65) node [anchor=north west][inner sep=0.75pt]  [font=\small]  {$\cdots $};
% Text Node
\draw (308.5,198.65) node [anchor=north west][inner sep=0.75pt]  [font=\small]  {$\cdots $};

\end{tikzpicture}

    \caption{Multilayer porous media}
    \label{fig1}
\end{figure}
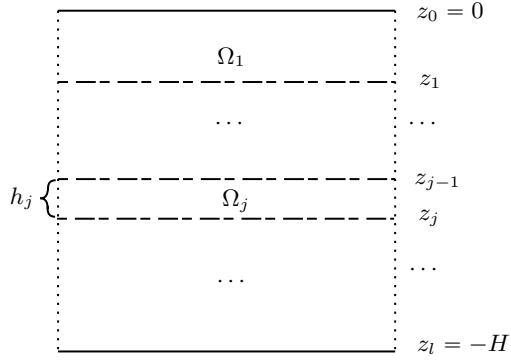

On the interfaces $z=z_i$, we assume the continuity of the normal velocity, the concentration density, and the pressure, i.e., 
\begin{equation}\label{interface-1}
  \Bu\cdot\Be_{z},\ \phi,\ P \text{ are continuous at } z=z_i,\ 1\leq  i\leq l-1.
\end{equation} 

Furthermore, system \eqref{Sharp} is supplemented with the initial condition
\begin{equation}\label{initialdata}
  \phi|_{t=0}=\phi_{0}
\end{equation}
and the boundary conditions
\begin{equation}\label{bc0}
\phi|_{z=0}=c_0, \quad \phi|_{z=-H}=c_{-H},~~\Bu\cdot \Be_z|_{z=0,-H}= 0,
\end{equation}
together with periodicity in the horizontal direction $x$.

The well-posedness of sharp interface model \eqref{Sharp}-\eqref{bc0} for $L^2$ initial data in 2D and $H^1$ initial data in 3D has been studied in \cite{CNW2025}. 

\section{Preliminary results}
Following the work of \cite{doering1998jfm}, we introduce a background profile to overcome the nonhomogeneous boundary condition of $\phi$ in \eqref{bc0}. More specifically,  we introduce a smooth function $\phi_b(z;\delta)$ on $[-H,0]$ satisfying
\begin{equation}\label{phi_b1}
\phi_b(z;\delta)= \left\{\begin{aligned}
& c_0, && z = 0,\\
& \frac{c_0+c_{-H}}{2}&&z \in (-H+\delta,-\delta),\\
& c_{-H}, && z=-H
\end{aligned}\right.
\end{equation}
and 
\begin{equation}\label{phi_b2}
|\phi_b'| \leq \frac{c_\Delta}{\delta}, ~~~~|\phi_b''| \leq \frac{2c_\Delta}{\delta^2},
\end{equation}
where $\delta>0$ is a parameter to be determined later and $c_\Delta := |c_0-c_{-H}|$. 

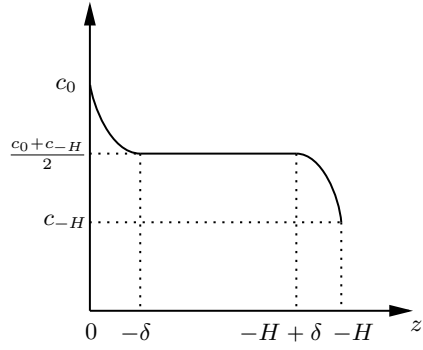
\begin{figure}[htbp]
    \centering

\tikzset{every picture/.style={line width=0.75pt}} %set default line width to 0.75pt        

\begin{tikzpicture}[x=0.75pt,y=0.75pt,yscale=-1,xscale=1]

%Straight Lines [id:da3815787109734916] 
\draw    (248.01,203) -- (409,203) ;
\draw [shift={(411,203)}, rotate = 180] [fill={rgb, 255:red, 0; green, 0; blue, 0 }  ][line width=0.08]  [draw opacity=0] (12,-3) -- (0,0) -- (12,3) -- cycle    ;
%Straight Lines [id:da04531581278878771] 
\draw    (248.01,203) -- (248.01,49) ;
\draw [shift={(248.01,47)}, rotate = 90] [fill={rgb, 255:red, 0; green, 0; blue, 0 }  ][line width=0.08]  [draw opacity=0] (12,-3) -- (0,0) -- (12,3) -- cycle    ;
%Straight Lines [id:da16270497607508194] 
\draw  [dash pattern={on 0.84pt off 2.51pt}]  (273.4,123.67) -- (273.4,203) ;
%Curve Lines [id:da4449994954442509] 
\draw    (248.13,89) .. controls (249.4,98.4) and (258.28,123.59) .. (273.4,123.67) ;
%Straight Lines [id:da6687210740002881] 
\draw  [dash pattern={on 0.84pt off 2.51pt}]  (374.8,158.33) -- (374.8,203) ;
%Straight Lines [id:da13932931117831482] 
\draw    (352.13,123.67) -- (273.4,123.67) ;
%Curve Lines [id:da769446766179348] 
\draw    (352.13,123.67) .. controls (365.09,123.64) and (373.8,145.6) .. (374.8,158.33) ;
%Straight Lines [id:da1062852127444176] 
\draw  [dash pattern={on 0.84pt off 2.51pt}]  (352.13,123.67) -- (352.13,203) ;
%Straight Lines [id:da25361938734326184] 
\draw  [dash pattern={on 0.84pt off 2.51pt}]  (247.8,158.33) -- (374.8,158.33) ;
%Straight Lines [id:da9389046037187144] 
\draw  [dash pattern={on 0.84pt off 2.51pt}]  (247.8,123.67) -- (273.4,123.67) ;

% Text Node
\draw (262,208.) node [anchor=north west][inner sep=0.75pt]  [font=\small]  {$-\delta $};
% Text Node
\draw (322,208) node [anchor=north west][inner sep=0.75pt]  [font=\small]  {$-H+\delta $};
% Text Node
\draw (229.13,84.4) node [anchor=north west][inner sep=0.75pt]  [font=\small]  {$c_{0}$};
% Text Node
\draw (223.13,153) node [anchor=north west][inner sep=0.75pt]  [font=\small]  {$c_{-H}$};
% Text Node
\draw (205,114) node [anchor=north west][inner sep=0.75pt]  [font=\small]  {$\frac{c_{0}+ c_{-H}}{2}$};
% Text Node
\draw (244.13,208.13) node [anchor=north west][inner sep=0.75pt]  [font=\small]  {$0$};
% Text Node
\draw (369,208) node [anchor=north west][inner sep=0.75pt]  [font=\small]  {$-H$};
% Text Node
\draw (408,206.4) node [anchor=north west][inner sep=0.75pt]  [font=\small]  {$z$};

\end{tikzpicture}

    \caption{The background profile $\phi_b$}
    \label{fig2}
\end{figure}
Setting
\[ \psi =\phi -  \phi_b  
\text{ and }  p = P -\rho_0 g z - \alpha \rho_0 \int_{-H}^z\phi_b(s)\dd s,\]
the system \eqref{Sharp} can be rewritten as
    \begin{equation}\label{Sharp-1}
      \left\{\begin{aligned}
        &\Bu=-\frac{K}{\mu}\left(\nabla p+  \alpha  \rho_0g \psi\Be_{z}\right),\\
        &\nabla \cdot \Bu =0,\\
        &b\partial_t \psi  + \Bu\cdot\nabla \psi   +  \phi_b'\Bu \cdot \Be_z   - \nabla\cdot(bD\nabla\psi)=bD \phi_b''.
      \end{aligned}\right. 
      \end{equation} 
For convenience, we will assume $\mu = 1, \rho_{0} \alpha g =1$ and $b=1$ throughout this paper. Then we obtain the simplified sharp interface model 
\begin{equation}\label{Sharps}
      \left\{\begin{aligned}
        &\Bu=- K \left(\nabla p  + \psi \Be_{z}\right),\\
        &\nabla \cdot \Bu =0,\\
        & \partial_t \psi + \Bu\cdot\nabla \psi  +  \phi_b'\Bu \cdot \Be_z   - \nabla\cdot(D\nabla\psi)=D \varphi_b'',
      \end{aligned}\right.  
  \end{equation} 
supplemented with the interfacial boundary conditions 
\begin{equation}\label{interface1}
  \Bu\cdot\Be_{z},\ \psi,\ p\text{ are continuous at } z=z_i,\ 1\leq i\leq l-1,
\end{equation} 
initial condition 
\begin{equation}\label{initial}
  \psi|_{t=0}=\psi_0:=\phi_{0} -\phi_b,
\end{equation}
homogeneous boundary condition 
\begin{equation}\label{bc}
    \psi|_{z=0,z=-H}=0,~~\Bu\cdot \Be_z|_{z=0,-H}= 0,
\end{equation}
and periodicity in the horizontal direction $x$.

While the lower-order terms in the revised  formulation \eqref{Sharps}-\eqref{bc} of the sharp interface model do not affect the well-posedness or finite-time behavior, they introduce additional difficulty in terms of the long-time behavior of the system.  The special background $\varphi_b$, which changes rapidly near the boundary and is constant in the interior of the domain, is chosen so that the linear non-symmetric term $ \phi_b'\Bu \cdot \Be_z$ can be controlled by part of the diffusion term (see Lemma \ref{L^2}).

For $1\leq r \leq \infty$ and $k\in \mathbb{R}$, let $L^r (\Omega)$ and $H^k(\Omega)$ denote the usual Lebesgue space of integrable functions and Sobolev spaces on the domain $\Omega$. The inner product in $L^2(\Omega)$ will be denoted by $(\cdot,\cdot)$. Define
\begin{equation}\nonumber
\mathcal{V}:= \{\psi\in C(\Omega),~\psi|_{\Omega_i} \in C^\infty(\overline{\Omega_i})\text{ for any } i=1,2,\cdots,l:~\psi \text{ satisfies }\eqref{interface1},\eqref{bc}\},
\end{equation} 
and $H, V$ be the closure of $\mathcal{V}$ in the $L^2$ and $H^1$ norm, respectively. We denote the dual space of $V$ by $V^*$. We also define  
\begin{equation}\nonumber
\widetilde{\mathcal{V}} := \{ \Bu \in C(\Omega),~\Bu|_{\Omega_i} \in C^\infty(\overline{\Omega_i})\text{ for any } i=1,2,\cdots,l:~\Bu \text{ satisfies }\eqref{Sharps}_2, \eqref{interface1}, \eqref{bc}\},
\end{equation} 
and $\boldsymbol{H}$ be the closure of $\widetilde{\mathcal{V}}$ in the $L^2$ norm. Then the weak solution of the sharp interface model \eqref{Sharps}-\eqref{bc} is defined as follows. 
\begin{definition}
Let $\psi_0\in  H$ be given, and let $T>0$. A weak solution of the problem \eqref{Sharps}-\eqref{bc} on the interval $[0,T]$ is a triple $(\Bu, \psi, p)$ satisfying the following conditions:
\begin{enumerate}
  \item $\psi \in L^\infty(0,T;H) \cap L^2(0,T;V)$, $\partial_t \psi \in L^2(0,T;V^*)$, 
  \begin{equation}\nonumber
    (\psi(t_2), \varphi) - ( \psi(t_1), \varphi) + \int_{t_1}^{t_2}(\Bu\cdot\nabla \psi, \varphi)  +  (\phi_b'\Bu \cdot \Be_z,\varphi)  + (D\nabla\psi,\nabla\varphi)-(D \varphi_b'',\varphi)\dd t=0,
  \end{equation}
  for any $t_1,t_2\in [0,T]$ and $\varphi\in V$, and $\psi(0)= \psi_0$,
  \item $p \in L^2(0,T;H^1(\Omega))$ is a weak solution to the elliptic problem
  \begin{equation}\label{pressure}
  \left\{\begin{aligned}
&-\nabla\cdot(K\nabla p) = \nabla\cdot (K\psi \Be_{z}) , ~~~~ \text{ in } \Omega,\\
&\partial_z p(x,-H)  = \partial_z p(x,0)  = 0,
  \end{aligned}\right. 
\end{equation}
% that is,
%   \begin{equation}\nonumber
%  (K\nabla p,\nabla q) = (K \psi\Be_z,\nabla q),~~~~ \text{ for any }q\in H^1(\Omega).
%   \end{equation}
  \item $\Bu \in L^2(0,T;\boldsymbol{H}(\Omega))$ satisfies \eqref{Sharps}$_1$.
\end{enumerate}
\end{definition}

The existence of the elliptic problem \eqref{pressure} for any $\psi\in H$ is guaranteed by the Lax-Milgram theorem. Uniqueness can be obtained if we restrict to the mean zero subspace of $H^1$. For piecewise smooth $\psi$, smooth in each layer $\Omega_i$, one can show that $p$ is also piecewise smooth. More specifically, if $\psi\in V$ , then $p$ is piecewise $H^2$. Moreover, one has the following $W^{1,r}$ estimate for the pressure (\cite{DL2021}).
\begin{lemma}\label{estimate-p}
If $\psi \in L^r(\Omega)$ for some $r\in (1,\infty)$, we have $\nabla p\in L^r(\Omega)$ and
\[\|\nabla p\|_{L^r(\Omega)} \leq C_p \|\psi\|_{L^r(\Omega)},\]
where $C_p = C_p(r, K_i,\Omega)$.
\end{lemma}

For higher regularity of the solutions, we recall the following space, which is introduced in \cite{CNW2025}.
\begin{definition}
Define
\begin{equation}\nonumber
W = \{\psi \in V:~\partial_x \psi \in H^1(\Omega),~D\partial_z \psi \in H^1(\Omega)\}
\end{equation}
endowed with norm
\begin{equation}\nonumber
  \|\psi\|_W = \|\psi\|_{H^1(\Omega)}+\|\partial_x \psi\|_{H^1(\Omega)}+\|D\partial_z \psi\|_{H^1(\Omega)}.
\end{equation}
\end{definition}
 The weighted space $W$ is different from the classical $H^2$ space in general unless unless $D$ is smooth in $\Omega$. The following lemma from \cite{CNW2025} shows that the space $W$ enjoys some Sobolev type inequalities similar to the $H^2$ space.   
\begin{lemma}\label{W-embedding}
 $W$ is similar to $H^2(\Omega)$ in the sense that the following inequalities hold
\begin{equation}
\|\nabla \psi\|_{L^r(\Omega)} \leq C\|\psi\|_W,~~\text{ for any }r\in [2,\infty).
\end{equation}
\end{lemma} 
Furthermore, one could also establish the Gagliardo-Nirenberg type inequality for the weighted space $W$. 
\begin{lemma}\label{lemmaA1}
  For any $\psi \in V$, it holds that
   \begin{equation}\label{A1-0}
  \|\psi\|_{L^4(\Omega)} \leq C_1\|\nabla\psi\|_{L^2(\Omega)}^\frac12\|\psi\|_{L^2(\Omega)}^\frac12 ,
  \end{equation}
Furthermore, we have 
  \begin{equation}\label{A1-1}
  \|\nabla \psi\|_{L^4(\Omega)} \leq C_2\|\nabla\psi\|_{L^2(\Omega)}^\frac12\|\psi\|_{W}^\frac12, ~~~~\text{ for any }\psi\in W.
  \end{equation}  
Here the constant  $C_1,C_2$ depends only on $\Omega$ and $D_i$.
\end{lemma}
\begin{proof}
For any $\psi \in V$, it follows from the classical Gagliardo-Nirenberg inequality that 
   \begin{equation}
  \|\psi\|_{L^4(\Omega)} \leq C(\|\nabla\psi\|_{L^2(\Omega)}^\frac12\|\psi\|_{L^2(\Omega)}^\frac12 + \|\psi\|_{L^2(\Omega)}).
  \end{equation}
  Since $\psi$ vanishes on the boundary $\partial \Omega$, one uses Poincar\'e inequality to obtain
     \begin{equation}
\|\psi\|_{L^2(\Omega)}\leq C(1+H^\frac12)\|\psi\|_{L^2(\Omega)}^\frac12\|\nabla\psi\|_{L^2(\Omega)}^\frac{1}{2} =: C_1 \|\psi\|_{L^2(\Omega)}^\frac12\|\nabla\psi\|_{L^2(\Omega)}^\frac{1}{2}
  \end{equation}
% This, together with \eqref{}, gives \eqref{A1-0}. 

Now we suppose that $\psi\in W$. Set $\tilde{\nabla}\psi := (\partial_x \psi, D\partial_z\psi)$. By definition, one has $\tilde{\nabla}\psi \in H^1(\Omega)$ and  
\begin{equation}
\|\tilde{\nabla}\psi\|_{H^1(\Omega)} = \|\partial_x \psi\|_{H^1(\Omega)}+ \|D\partial_z\psi\|_{H^1(\Omega)} \leq \|\psi\|_{W}.
\end{equation}
Applying the Gagliardo-Nirenberg inequality again to $\tilde{\nabla}\psi$ gives
\begin{equation}
\begin{aligned}
\|\tilde{\nabla}\psi\|_{L^4(\Omega)}\leq& ~C(\|\tilde{\nabla}\psi\|_{L^2(\Omega)}^\frac12\|\nabla(\tilde{\nabla}\psi)\|_{L^2(\Omega)}^\frac12+\|\tilde{\nabla}\psi\|_{L^2(\Omega)})\\
\leq& ~C(\|\tilde{\nabla}\psi\|_{L^2(\Omega)}^\frac12\|\tilde{\nabla}\psi\|_{H^1(\Omega)}^\frac12+ \|\tilde{\nabla}\psi\|_{L^2(\Omega)})\\
\leq& ~2C\|\tilde{\nabla}\psi\|_{L^2(\Omega)}^\frac12\|\tilde{\nabla}\psi\|_{H^1(\Omega)}^\frac12\\
\leq & ~2C\|\tilde{\nabla}\psi\|_{L^2(\Omega)}^\frac12\|\psi\|_{W}^\frac12.
\end{aligned}
\end{equation}
Noting 
\begin{equation}
 \min_i D_i \|\nabla\psi\|_{L^p(\Omega)} \leq \|\tilde{\nabla}\psi\|_{L^p(\Omega)} \leq \max_i D_i \|\nabla\psi\|_{L^p(\Omega)} ,
\end{equation}
for any $p\ge 1$, one concludes \eqref{A1-1} with $C_2 = \frac{2C\max_i D_i}{\min_i D_i}$. This finishes the proof. 
\end{proof}

We also recall the equivalence between $\|\cdot\|_W$ and the norm associated with the operator $\mathcal{L}\psi:= -\nabla\cdot (D\psi)$. One may refer to \cite{CNW2025} for the detailed proof.
\begin{lemma}\label{lemma-equivalence}
 There exist two constants $C_l, C_u>0$ depending only on $H$ and $D_i$ such that
\begin{equation} \label{A4-0-1}
C_l\|\mathcal{L}\psi\|_{L^2(\Omega)} \leq \|\psi\|_{W} \leq C_u\|\mathcal{L}\psi\|_{L^2(\Omega)},~~~\text{ for any }\psi \in W.
\end{equation}  
\end{lemma}

Following \cite[Theorem 3.1, 4.1]{CNW2025}, one can obtain the existence, uniqueness, and regularity of solutions to the sharp interface model \eqref{Sharps}-\eqref{bc} as follows.
\begin{proposition}\label{prop1}
For each $\psi_0\in H$, there exists a unique global weak solution $(\Bu,\psi, p)$ to problem \eqref{Sharps}-\eqref{bc}. In particular, the map $\psi_0\mapsto \psi(t)$ is continuous in $\psi_0$ and $t$ in $H$. If, furthermore, $\psi_0\in V$,  we have
\begin{equation}
\psi \in C([0,T);V) \cap L^2(0,T;W),~~~~~ \text{for any }T>0.
\end{equation}
\end{proposition}
% The title of your section 3:

%%%%% 4
\section{Uniform estimates on the sharp interface model with a thin layer} 

In this section, we establish the uniform $H^1$-estimates for the solution to the problem \eqref{Sharps}-\eqref{bc}. These long time uniform in $\varepsilon$ estimates stated in Propositions \ref{L^2}-\ref{H^1} guarantee the existence of a uniform bounded absorbing set  in $V$, a crucial component of the existence and convergence of global attractors. Furthermore, the finite-time uniform estimates given in Proposition \ref{finite-time} play a crucial role in our subsequent proof of the finite-time convergence of the solution as the thickness of the layer tends to zero.

\begin{proposition}\label{L^2}
Let $\delta$ be chosen so that \eqref{delta_2} is satisfied.
Suppose that $\psi_0\in H$ and $(\Bu,\psi,p)$ is a solution to the problem \eqref{Sharps}-\eqref{bc}.  It holds that 
\begin{equation}\label{L^2-1}
\|\psi(t)\|_{L^2(\Omega)}^2 \leq \|\psi_0\|_{L^2(\Omega)}^2 e^{-\frac{\min_i D_i}{H^2}t} +\frac{M_1H^2}{\min_i D_i}\left(1- e^{-\frac{\min_i D_i}{H^2}t}\right),
\end{equation}
and 
\begin{equation}\label{L^2-1-1}
\int_0^t \|\sqrt{D} \nabla \psi(s)\|_{L^2(\Omega)}^2\dd s  \leq M_1 t +\| \psi_0\|_{L^2(\Omega)}^2,
\end{equation}
for any $t>0$, where
\begin{equation}\label{M1}
M_1= \frac{8 c_\Delta^2 L }{\delta}\frac{ (\max_i D_i )^2}{\min_i  D_i }.
\end{equation}

In particular, for any  $t\ge T_1:=T_1(\|\psi_0\|_{L^2(\Omega)}) =\frac{2H^2}{\min_i D_i}\ln \|\psi_0\|_{L^2(\Omega)}$, one has 
\begin{equation}\label{L^2-1-2}
\|\psi(t) \|_{L^2(\Omega)}^2\leq  \frac{M_1H^2}{\min_i D_i}+1,~~~~~~\int_t^{t+1} \|\sqrt{D}\nabla\psi \|_{L^2(\Omega)}^2 \dd s \leq  \frac{M_1H^2 }{\min_i D_i } +1.
\end{equation}
\end{proposition}

\begin{proof}
Testing the equation \eqref{Sharps}$_3$ by $\psi$ and using integration by parts yield  
\begin{equation}\label{L^2-2}
\frac12  \frac{\dd}{\dd t} \|\psi\|_{L^2(\Omega)}^2 + \|\sqrt{D}\nabla\psi\|_{L^2(\Omega)}^2 =  -   \int_\Omega  \phi_b' u_z \psi \dd \Bx + \int_\Omega D \varphi_b''\psi \dd \Bx .
\end{equation} 
By virtue of \eqref{phi_b1} and \eqref{phi_b2},  we have 
\begin{equation}\label{L^2-3}
\left|\int_\Omega  \phi_b' u_z \psi \dd \Bx\right| = \left|\int_{\Omega_\delta}  \phi_b' u_z \psi \dd \Bx\right| \leq  c_\Delta \delta^{-1}\|u_z\|_{L^2(\Omega_\delta)} \|\psi\|_{L^2(\Omega_\delta)},
\end{equation} 
where $\Omega_\delta := \Omega \cap (\{-H<z<-H+\delta\} \cup \{-\delta <z<0\})$. Noting that both $u_z$ and $\psi$ vanish on  the boundary $\{z=0\}\cup \{z=-H\}$, one uses Poincar\'e's inequality to obtain
\begin{equation} \label{L^2-3-1}
\|u_z\|_{L^2(\Omega_\delta)} \leq   \delta\|\partial_z u_z\|_{L^2(\Omega_\delta)},~~\|\psi\|_{L^2(\Omega_\delta)} \leq   \delta\|\partial_z \psi\|_{L^2(\Omega_\delta)},
\end{equation}
and thus,
\begin{equation}\label{L^2-4}
\left|\int_\Omega  \phi_b' u_z \psi \dd \Bx\right|\leq c_\Delta\delta\|\partial_z u_z\|_{L^2(\Omega_\delta)}\|\partial_z \psi\|_{L^2(\Omega_\delta)} \leq  c_\Delta\delta \|\partial_x u_x\|_{L^2(\Omega)}\|\partial_z \psi\|_{L^2(\Omega)}.
\end{equation} 
Here we have used the fact that $\partial_z u_z = - \partial_x u_x$ due to the incompressibility condition $\nabla \cdot \Bu =0$. In view of \eqref{Sharps}$_1$, we have
\begin{equation}\label{L^2-5}
\|\partial_x \Bu\|_{L^{2}(\Omega)} \leq \max_i  K_i(\|\partial_x \nabla p\|_{L^{2}(\Omega)}+ \|\partial_x  \psi\|_{L^2(\Omega)}),
\end{equation}
while the pressure $p$ is determined by the elliptic problem \eqref{pressure}. Since the permeability coefficient $K$ is independent of $x$, one may take the derivative of \eqref{pressure} with respect to $x$ and obtain 
 \begin{equation}\nonumber\label{pressure_x}
  \left\{\begin{aligned}
&-\nabla\cdot(K \nabla  \partial_x p) = \nabla\cdot (K \partial_x \psi \Be_z) ,  ~~~~\text{ in } \Omega,\\
&\partial_z \partial_x p(x,-H)  = \partial_z \partial_x p(x,0)  = 0.
  \end{aligned}\right. 
\end{equation}
Then it follows from the standard elliptic estimate  that we have
\begin{equation}\nonumber
\|\partial_x\nabla p\|_{L^2(\Omega)} \leq   \frac{\max_i  K_i}{\min_i  K_i}\|\partial_x\psi\|_{L^2(\Omega)}.
\end{equation}
This, together with \eqref{L^2-5}, yields
\begin{equation}\nonumber
\|\partial_x \Bu\|_{L^{2}(\Omega)} \leq \max_i  K_i(\|\partial_x \nabla p\|_{L^{2}(\Omega)}+ \|\partial_x  \psi\|_{L^2(\Omega)})\leq  \frac{2(\max_i  K_i)^2}{\min_i  K_i} \|\partial_x\psi\|_{L^2(\Omega)}.
\end{equation}
Hence,
\begin{equation}\label{L^2-6}
\left|\int_\Omega  \phi_b' u_z \psi \dd \Bx \right|\leq c_\Delta\delta \|\partial_x u_x\|_{L^2(\Omega)}\|\partial_z \psi\|_{L^2(\Omega)}\leq  \frac{2(\max_i  K_i)^2 c_\Delta \delta }{\min_i  K_i} \|\nabla \psi\|_{L^2(\Omega)}^2.
\end{equation}
On the other hand, using H\"older's inequality, Young's inequality and \eqref{L^2-3-1}, we have
\begin{equation}\label{L^2-7}
\begin{aligned}
\left|\int_\Omega  D \varphi_b''\psi \dd \Bx\right| \leq  &~ 2c_\Delta \delta^{-2} \max_i D_i  |\Omega_\delta|^\frac12 \| \psi\|_{L^2(\Omega_\delta)}  \\
\leq  &~ 2c_\Delta L^\frac12 \delta^{-\frac12} \max_i D_i    \| \partial_z\psi\|_{L^2(\Omega_\delta)} \\ 
\leq &~ \frac{1}{4}\min_i  D_i \|\nabla \psi\|_{L^2(\Omega)}^2+  \frac{4 c_\Delta^2 L}{\delta}\frac{ (\max_i D_i )^2}{\min_i  D_i }.
\end{aligned}
\end{equation}

Substituting \eqref{L^2-6}-\eqref{L^2-7} into \eqref{L^2-2} and choosing 
\begin{equation}\label{delta_2}
\delta \leq  \frac{\min_i  K_i \min_i D_i}{8(\max_i  K_i)^2 c_\Delta},
\end{equation} 
one has 
\begin{equation}\label{L^2-8}
 \frac{\dd}{\dd t} \|\psi\|_{L^2(\Omega)}^2  +   \|\sqrt{D}\nabla\psi\|_{L^2(\Omega)}^2 \leq  \frac{8 c_\Delta^2 L }{\delta}\frac{ (\max_i D_i )^2}{\min_i  D_i }=:M_1.
\end{equation}  
Integrating \eqref{L^2-8} over $[0,T]$ gives \eqref{L^2-1-1}. Furthermore, due to Poincar\'e inequality, we also infer from \eqref{L^2-8} that 
\begin{equation}\nonumber
 \frac{\dd}{\dd t} \|\psi\|_{L^2(\Omega)}^2  + \frac{\min_i D_i}{H^2} \|\psi\|_{L^2(\Omega)}^2 \leq M_1.
\end{equation}  
Using the classical  Gr\"onwall's inequality, we get \eqref{L^2-1}.  In particular, it follows from \eqref{L^2-1} that
\begin{equation}\label{L^2-9}
\|\psi(t)\|_{L^2(\Omega)}^2 \leq \frac{M_1H^2}{\min_iD_i}+1,~~~~\text{ for any }t \ge T_1:=\frac{2H^2}{\min_iD_i}\ln  \|\psi_0\|_{L^2(\Omega)}.
\end{equation}
Finally, for $t\ge T_1$, we integrate the energy inequality \eqref{L^2-8} over $[t,t+1]$ to obtain
\begin{equation}\label{L^2-10}
\|\psi(t+1) \|_{L^2(\Omega)}^2+  \int_t^{t+1} \|\sqrt{D}\nabla\psi \|_{L^2(\Omega)}^2 \dd s \leq \|\psi(t) \|_{L^2(\Omega)}^2 \leq \frac{M_1H^2}{\min_i D_i}+1.
\end{equation}
Combining \eqref{L^2-9} and \eqref{L^2-10}, one proves \eqref{L^2-1-2} and finishes the proof of the proposition.
\end{proof}

\begin{proposition}\label{H^1}
Suppose that $\psi_0\in H$ and the triple $(\Bu,\psi, p)$ is a solution to the problem \eqref{Sharps}-\eqref{bc}. Then for any $t\ge T_1+1$, it holds that
\begin{equation} \label{2-14}
\|\sqrt{D}\nabla \psi(t)\|_{L^2(\Omega)}^2   \leq  M_5.
\end{equation}
where the constant $M_5$ is given in \eqref{M5}.
\end{proposition}
\begin{proof}
Multiplying \eqref{Sharps}$_3$ by $\mathcal{L}\psi := -\nabla\cdot (D\nabla\psi)$ and integrating result equation over $\Omega$, we obtain
\begin{equation}\label{H^1-3}
\frac12  \frac{\dd}{\dd t} \|\sqrt{ D}\nabla \psi\|_{L^2(\Omega)}^2 + \|\mathcal{L}\psi\|_{L^2(\Omega)}^2 =  (\Bu\cdot\nabla  \psi,\mathcal{L}\psi) + (\phi_b' u_z, \mathcal{L}\psi ) -  (D\varphi_b'',\mathcal{L}\psi)   .
\end{equation}
Using H\"older's inequality, Young's inequality and \eqref{phi_b2}, one has 
\begin{equation}\label{H^1-4}
\begin{aligned}
|(D \varphi_b'',\mathcal{L}\psi)|\leq  &~ \max_i D_i \frac{2c_\Delta}{\delta^2}\int_{\Omega_{\delta}} | \mathcal{L}\psi|  \dd \Bx  \\
\leq  &~ \max_i D_i \frac{2c_\Delta}{\delta^2} |\Omega_\delta|^\frac12 \|\mathcal{L}\psi\|_{L^2(\Omega_\delta)}  \\
\leq  &\frac{1}{4}\|\mathcal{L}\psi\|_{L^2(\Omega)}^2 +4 c_\Delta^2 L \delta^{-3} (\max_i D_i )^2.
\end{aligned}
\end{equation}
In a way similar to \eqref{L^2-3}-\eqref{L^2-6}, one may use Young's inequality to deduce
\begin{equation}\label{H^1-5}
\begin{aligned}
\left|(\phi_b' u_z \mathcal{L}\psi)\right| 
\leq  &~ c_\Delta \delta^{-1} \|u_z\|_{L^2(\Omega_\delta)}\|\mathcal{L}\psi\|_{L^2(\Omega_\delta)}\leq  c_\Delta \|\partial_x u_x\|_{L^2(\Omega)}\|\mathcal{L}\psi\|_{L^2(\Omega)}\\
\leq  &~   \frac{2(\max_i  K_i)^2 c_\Delta  }{\min_i  K_i} \|\partial_x \psi\|_{L^2(\Omega)}\|\mathcal{L}\psi\|_{L^2(\Omega)}\\
\leq  &~ \frac{1}{4}\|\mathcal{L}\psi\|_{L^2(\Omega)}^2+\frac{4(\max_i  K_i)^4 c_\Delta^2  }{(\min_i  K_i)^2}\|\nabla\psi\|_{L^2(\Omega)}^2.
\end{aligned}
\end{equation} 
Furthermore, using H\"older's inequality, Young's inequality and Lemma \ref{lemmaA1}, we have 
\begin{equation}\label{H^1-6}
\begin{aligned}
\left|(\Bu\cdot\nabla \psi, \mathcal{L}\psi)\right| 
\leq  &~ \|\Bu\|_{L^4(\Omega)} \|\nabla \psi\|_{L^4(\Omega)}  \|\mathcal{L} \psi\|_{L^2(\Omega)}\\
\leq  &~ C_2 C_u^\frac12\|\Bu\|_{L^4(\Omega)} \|\nabla \psi\|_{L^2(\Omega)}^\frac12 \|\mathcal{L} \psi\|_{L^2(\Omega)}^\frac32\\ 
\leq  &~\frac{1}{4}\|\mathcal{L} \psi\|_{L^2(\Omega)}^2+  \frac{27}{4}C_2^2 C_u^2\|\Bu\|_{L^4(\Omega)}^4 \|\nabla \psi\|_{L^2(\Omega)}^2.
\end{aligned}
\end{equation} 
By virtue of \eqref{Sharps}$_1$, Lemma \ref{lemmaA1} and Poincar\'e's inequality, it holds that
\begin{equation}\label{H^1-7}
\begin{aligned}
  \|\Bu\|_{L^4(\Omega)} \leq &~\max_i K_i(\|\nabla p\|_{L^4(\Omega)}+\|\psi\|_{L^4(\Omega)})\\ 
\leq  &  (1+C_p) \max_i K_i \|\psi\|_{L^4(\Omega)} \\ 
\leq  &C_1  (1+C_p) \max_i K_i \|\psi\|_{L^2(\Omega)}^\frac12 \|\nabla \psi\|_{L^2(\Omega)}^\frac12.
\end{aligned}
\end{equation} 
%where $C(4,K_i)$ is the constant appeared in Lemma \ref{estimate-p}.
Substituting the estimates \eqref{H^1-4}-\eqref{H^1-8} into \eqref{H^1-3}, one obtains
\begin{equation}\label{H^1-8}
\begin{aligned}
&~\frac{\dd}{\dd t} \|\sqrt{D}\nabla \psi\|_{L^2(\Omega)}^2  + \frac{1}{2} \|\mathcal{L}\psi\|_{L^2(\Omega)}^2 \\
\leq&~ M_2 + (M_3  + M_4\|\psi\|_{L^2(\Omega)}^2 \|\sqrt{D}\nabla \psi\|_{L^2(\Omega)}^2)\|\sqrt{D} \nabla \psi\|_{L^2(\Omega)}^2  ,
\end{aligned}
\end{equation} 
with constants
 \begin{equation}\label{M2M3}
  M_2= 8 c_\Delta^2 L \delta^{-3} (\max_i D_i )^2,~~M_3 = \frac{8(\max_i  K_i)^4 c_\Delta^2  }{(\min_i  K_i)^2\min_i  D_i},
\end{equation}
and 
\begin{equation}\label{M4}
M_4 = \frac{27}{2}C_1^4C_2^2 C_u^2(1+C_p)^4\frac{(\max_i K_i)^4}{(\min_i D_i)^2}.
\end{equation}   

Set
\begin{equation}\nonumber
y(t) := \|\sqrt{D}\nabla \psi\|_{L^2(\Omega)}^2, ~~~g(t) := M_3+M_4\|\psi\|_{L^2(\Omega)}^2 \|\sqrt{D}\nabla \psi\|_{L^2(\Omega)}^2,~~~h(t) := M_2.
\end{equation}
Then, for any $t>0$, we have
\begin{equation}\label{2-20}
  y'(t) \leq g(t) y(t) + h(t) 
\end{equation}
and 
\begin{equation}
\int_{t}^{t+1} h(s)\dd s = M_2.
\end{equation}
Furthermore, if $t\ge T_1$, it follows from Proposition \ref{L^2} that
\begin{equation}\nonumber
  \int_{t}^{t+1} y(s)\dd s = \int_{t}^{t+1} \|\sqrt{D}\nabla \psi\|_{L^2(\Omega)}^2 \dd s \leq \frac{M_1H^2 }{\min_i D_i} +1
\end{equation} 
and 
\begin{equation}\nonumber
\begin{aligned}
      \int_{t}^{t+1} g(s)\dd s \leq&~ M_3 +M_4\sup_{s\in[t,t+1]} \|\psi(s)\|_{L^2(\Omega)}^2\int_{t}^{t+1}   \|\sqrt{D}\nabla \psi\|_{L^2(\Omega)}^2 \dd s \\
      \leq&~ M_3+ M_4 \left(\frac{M_1H^2}{\min_i D_i}+1\right)^2.
\end{aligned}
\end{equation} 
Then an application of the uniform Gr\"onwall lemma \cite[Lemma 1.1]{Temam1997} gives
\begin{equation}\label{M5}
y(t+1)  \leq  \left(M_2+  \frac{M_1H^2}{\min_i D_i}+1\right) \operatorname{exp}\left[M_3+M_4\left(\frac{M_1H^2}{\min_i D_i}+1\right)^2\right]=:M_5,
\end{equation}
for any $t\ge T_1$. This finish the proof.
\end{proof}

Furthermore, given initial data $\psi_0$ with higher regularity, we obtain the following finite-time estimate.
\begin{proposition}\label{finite-time}
Suppose that $\psi_0\in V$ and the triple $(\Bu,\psi, p)$ is a solution to the problem \eqref{Sharps}-\eqref{bc}. Then for any $t>0$, it holds that
\begin{equation}\label{H^1-1} 
  \|\nabla \psi(t)\|_{L^2(\Omega)} ^2 \leq M_6(\|\psi_0\|_{L^2(\Omega)},\|\sqrt{D}\nabla\psi_0\|_{L^2(\Omega)})
\end{equation}
and 
\begin{equation}\label{H^1-2}
  \int_0^t \|\mathcal{L} \psi(s)\|_{L^2(\Omega)}^2\dd s \leq M_7(t,\|\psi_0\|_{L^2(\Omega)},\|\sqrt{D}\nabla\psi_0\|_{L^2(\Omega)}),
\end{equation} 
where the constants $M_6$ and $M_7$ are given in \eqref{M6} and \eqref{M7}, respectively.
\end{proposition}
\begin{proof}
Following the proof of Proposition \ref{H^1}, one obtains also the energy inequality \eqref{H^1-8}. With the aid of \eqref{L^2-1} and \eqref{L^2-1-1}, we could apply the  Gr\"onwall's inequality to \eqref{H^1-8} and deduce
\begin{equation}\nonumber
\begin{aligned}
 &~\|\sqrt{D}\nabla \psi(t)\|_{L^2(\Omega)}^2 \\
 \leq&~ \|\sqrt{D}\nabla \psi_0\|_{L^2(\Omega)}^2\exp\left(\int_0^t (M_3+M_4\|\psi\|_{L^2(\Omega)}^2 \|\sqrt{D}\nabla \psi\|_{L^2(\Omega)}^2) \dd s\right) \\
 &~+ \int_0^t M_2  \exp\left(\int_\tau^t (M_3+M_4\|\psi\|_{L^2(\Omega)}^2 \|\sqrt{D}\nabla \psi\|_{L^2(\Omega)}^2) \dd s\right) \dd \tau\\
 \leq&~ (\|\sqrt{D}\nabla \psi_0\|_{L^2(\Omega)}^2+  M_2 t)\exp\left(M_3 t +M_4\int_0^t \|\psi\|_{L^2(\Omega)}^2 \|\sqrt{D}\nabla \psi\|_{L^2(\Omega)}^2  \dd s\right)\\
 \leq&~ (\|\sqrt{D}\nabla \psi_0\|_{L^2(\Omega)}^2+ M_2 t)\\
 &\quad\cdot\exp\left[M_3 t +M_4\left(\frac{M_1H^2}{\min_i D_i}+\|\psi_0\|_{L^2(\Omega)}^2\right)\left(M_1t + \|\psi_0\|_{L^2(\Omega)}^2\right)\right]\\
 =&:M_5'(t,\|\psi_0\|_{L^2(\Omega)},\|\sqrt{D}\nabla\psi_0\|_{L^2(\Omega)}).
 \end{aligned}
\end{equation} 
This, together with Proposition \ref{H^1}, gives  
\begin{equation}\label{M6}
\begin{aligned}
\|\sqrt{D}\nabla \psi(t)\|_{L^2(\Omega)}^2\leq&~ \max\{M_5,M_5'(T_1+1,  \|\psi_0\|_{L^2(\Omega)},\|\sqrt{D}\nabla\psi_0\|_{L^2(\Omega)})\}  \\
=&:M_6(\|\psi_0\|_{L^2(\Omega)},\|\sqrt{D}\nabla\psi_0\|_{L^2(\Omega)}),
\end{aligned}
\end{equation}
for any $t> 0$.  Finally, we integrate \eqref{H^1-8} over $[0,t]$ to obtain 
\begin{equation}\label{M7}
\begin{aligned}
&~\frac12\int_0^t \|\mathcal{L}\psi\|_{L^2(\Omega)}^2 \dd s  \\
  \leq&~  \int_0^t M_2 + (M_3+M_4\|\psi\|_{L^2(\Omega)}^2 \|\sqrt{D}\nabla \psi\|_{L^2(\Omega)}^2)  \|\sqrt{D}\nabla \psi\|_{L^2(\Omega)}^2\dd s\\
  &~+\|\sqrt{D}\nabla \psi_0\|_{L^2(\Omega)}^2\\
  \leq&~ \left[M_2  + M_3M_6   + M_4 M_6^2   \left(\frac{M_1H^2}{\min_i D_i}+\|\psi_0\|_{L^2(\Omega)}^2\right)\right] t  +\|\sqrt{D}\nabla \psi_0\|_{L^2(\Omega)}^2\\
  =:&~\frac12 M_7(t,\|\psi_0\|_{L^2(\Omega)},\|\sqrt{D}\nabla\psi_0\|_{L^2(\Omega)}).
\end{aligned}
\end{equation} 
Then we finish the proof.
\end{proof}

\section{Convergence} 
In this section, we study the convergence of the solutions of the problem \eqref{Sharps}-\eqref{bc} as  the thicknesses of some layer goes to zero. Our convergence analysis proceeds in two steps: first, with the aid of the finite-time uniform estimates given in Proposition \ref{finite-time}, Theorem \ref{convergence} establishes the convergence on arbitrary time intervals; this, together with the attracting property of the dynamical system,  yields the convergence of the long-time state (see Theorem \ref{thm1}).

For convenience, we denote the thicknesses of the $(j-1)$-th and $j$-th layers by $h-\varepsilon$ and $\varepsilon$, respectively, while the thicknesses of the other layers are assumed to be fixed constants. The corresponding permeability and diffusion coefficient are denoted by $K^\varepsilon$ and $D^\varepsilon$, respectively. In the limit case  $\varepsilon=0$, we have $z_{j-1}=z_{j}$, and then the parameters $K^0, D^0$ are defined in the layered domain with one fewer layer, that is, 
\[ K^0(\Bx) = K_i(z),~D^0(\Bx) = D_i(z),\quad\text{ if }z \in (z_{i-1},z_i) \text{ for } i=1,\cdots,j-1,j+1,\cdots,l.\]

\subsection{Finite-time convergence} For any $\varepsilon$, let $(\Bu^\varepsilon,\psi^\varepsilon,p^\varepsilon)$ be the solution to the sharp interface problem \eqref{Sharps}-\eqref{bc} with thin layer. Denote
\[\tilde{\Bu} = \Bu^0-\Bu^\varepsilon,~\tilde{p} = p^0-p^\varepsilon,~\tilde{\psi} = \psi^0-\psi^\varepsilon,\]
and 
\[\tilde{K} := K^0-K^\varepsilon,\quad\tilde{D} := D^0 - D^\varepsilon.\]
Then $(\tilde{\Bu},\tilde{\psi},\tilde{p})$ satisfies the following system
\begin{equation}\label{error} 
      \left\{\begin{aligned}
        &\tilde{\Bu}=- K ^\varepsilon\left(\nabla \tilde{p } + \tilde{\psi} \Be_{z}\right)- \tilde{K} \left(\nabla p^0   + \psi^0  \Be_{z}\right),\\
        &\nabla \cdot \tilde{\Bu} =0,\\
        & \partial_t \tilde{\psi} + \Bu^\varepsilon\cdot\nabla \tilde{\psi}+\tilde{\Bu}\cdot\nabla \psi^0   +  \phi_b'\tilde{\Bu} \cdot \Be_z   =  \nabla\cdot(D^\varepsilon\nabla \tilde{\psi}) + \nabla\cdot(\tilde{D} \nabla \psi^0 )+ \tilde{D}\phi_b''.
      \end{aligned}\right.  
  \end{equation}   

\begin{theorem}\label{convergence}
Let $\psi_0 = \psi_0^\varepsilon \in V$. Then, as $\varepsilon$ goes to $ 0^+$, the corresponding solution $(\Bu^\varepsilon, \psi^\varepsilon, \nabla p^\varepsilon)$ converges to $(\Bu^0, \psi^0, \nabla p^0)$ in $L^\infty(0,T;L^2(\Omega))$ for any fixed $T$. Furthermore, the difference $(\tilde{\Bu},\tilde{\psi},\nabla \tilde{p})$ satisfies
  \begin{equation}\nonumber
\|\tilde{\Bu}\|_{L^2(\Omega)}^2+ \| \tilde{\psi}\|_{L^2(\Omega)}^2+ \| \nabla\tilde{p}\|_{L^2(\Omega)}^2\leq M_8\varepsilon^\frac14,~~~~\text{for any } t>0,
  \end{equation}
  where the constant $M_8= M_8(t,\|\psi_0\|_{L^2(\Omega)},\|\sqrt{D}\nabla\psi_0\|_{L^2(\Omega)})$ is given in \eqref{M8}. 
\end{theorem}
\begin{proof}
First, we note that the difference $\tilde{K},\tilde{D}$ are piecewise constant functions supported in a narrow band of width $\varepsilon$. Thus,
\begin{equation}\label{convergence-1}
\|\tilde{K}\|_{L^r(\Omega)} \leq |K_{j-1}-K_j| \varepsilon^{\frac{1}{r}}, \quad \|\tilde{D}\|_{L^r(\Omega)} \leq |D_{j-1}-D_j| \varepsilon^{\frac{1}{r}},
\end{equation}
for any $r\in [1,\infty]$.

The difference $\tilde{p}$ of pressure is determined by the elliptic problem
  \begin{equation}\label{error-pressure}
  \left\{\begin{aligned}
&-\nabla\cdot(K^\varepsilon\nabla {\tilde{p}}) = \nabla\cdot (K^\varepsilon\tilde{\psi}\Be_{z} +\tilde{K}(\nabla p^0 +\psi^0  \Be_z)) , ~~~~ \text{ in } \Omega,\\
&\partial_z \tilde{p}(x,-H)  = \partial_z \tilde{p}(x,0)  = 0.
  \end{aligned}\right. 
\end{equation}
Then it follows from Lemma \ref{estimate-p} and \eqref{convergence-1} that 
\begin{equation}\nonumber
\begin{aligned}
\|\nabla \tilde{p}\|_{L^4(\Omega)} \leq &~C (\|K^\varepsilon\tilde{\psi}\Be_{z} \|_{L^4(\Omega)}+\|\tilde{K}\nabla p^0\|_{L^4}+\|\tilde{K} \psi^0\|_{L^4(\Omega)})\\
\leq&~C(\|K^\varepsilon\|_{L^\infty(\Omega)}\|  \tilde{\psi}\|_{L^4(\Omega)}+\|\tilde{K}\|_{L^8(\Omega)}(\|\nabla p^0 \|_{L^8(\Omega)}+\|\psi^0 \|_{L^8(\Omega)}))\\
\leq&~C \|\tilde{\psi}\|_{L^4(\Omega)}+ C \varepsilon^\frac18 \|\psi^0 \|_{L^8(\Omega)}.
\end{aligned}
\end{equation}
According to Sobolev imbedding inequality and Proposition \ref{finite-time}, one has
\begin{equation}\nonumber
\|\nabla \tilde{p}\|_{L^4(\Omega)} 
\leq ~C \|\tilde{\psi}\|_{L^4(\Omega)}+ C \varepsilon^\frac18 \|\nabla\psi^0\|_{L^2(\Omega)} \leq C \|\tilde{\psi}\|_{L^4(\Omega)}+ CM_6^\frac12 \varepsilon^\frac18.
\end{equation}
Consequently, 
\begin{equation}\label{convergence-2}
\begin{aligned}
\|\tilde{u}\|_{L^4(\Omega)} \leq&~ \|K ^\varepsilon \|_{L^\infty(\Omega)}(\|\nabla \tilde{p }\|_{L^4(\Omega)}+ \|\tilde{\psi}\|_{L^4(\Omega)} ) \\
&~+ \|\tilde{K}\|_{L^8(\Omega)} (\|\nabla p^0 \|_{L^8(\Omega)}  + \|\psi^0 \|_{L^8(\Omega)})\\
\leq&~C (\| \tilde{\psi}\|_{L^4(\Omega)}+ M_6^\frac12  \varepsilon^\frac18  )  + C \varepsilon^\frac18 \|\psi^0 \|_{L^8(\Omega)}\\
\leq&~C (\| \tilde{\psi}\|_{L^4(\Omega)}+ M_6^\frac12  \varepsilon^\frac18  )  + C \varepsilon^\frac18 \|\nabla\psi^0 \|_{L^2(\Omega)}\\
\leq&~C (\| \tilde{\psi}\|_{L^4(\Omega)}+ M_6^\frac12  \varepsilon^\frac18  ).
\end{aligned}
\end{equation}
In a similar way, one can also prove 
\begin{equation}\label{convergence-3-1}
\|\nabla\tilde{p}\|_{L^2(\Omega)} \leq C (\| \tilde{\psi}\|_{L^2(\Omega)}+ M_6^\frac12  \varepsilon^\frac14 ).
\end{equation}
and
\begin{equation}\label{convergence-3}
\|\tilde{u}\|_{L^2(\Omega)} \leq C (\| \tilde{\psi}\|_{L^2(\Omega)}+ M_6^\frac12  \varepsilon^\frac14 ).
\end{equation}

Next, testing \eqref{error}$_3$ by $ \tilde{\psi}$ and using integration by parts, we obtain
\begin{equation}\label{convergence-8}
\begin{aligned}
&~\frac12\frac{\dd }{\dd t}\|\tilde{\psi}\|_{L^2(\Omega)}^2  + \|\sqrt{D^\varepsilon} \nabla \tilde{\psi}\|_{L^2(\Omega)}^2 \\
= &~(\tilde{\Bu}\cdot\nabla   \tilde{\psi}, \psi^0 ) -  (\phi_b'\tilde{\Bu} \cdot \Be_z, \tilde{\psi}) - (\tilde{D} \nabla \psi^0 ,\nabla\tilde{\psi}) +  (\tilde{D}\phi_b'', \tilde{\psi}).
\end{aligned}
\end{equation}

With the aid of \eqref{convergence-2}-\eqref{convergence-3}, one applies Sobolev embedding inequality and Lemma \ref{lemmaA1} to obtain
\begin{equation}\label{convergence-4}
\begin{aligned}
|(  \tilde{\Bu}\cdot\nabla  \tilde{\psi}  ,\psi^0)| \leq &~ \|\nabla  \tilde{\psi}\|_{L^2(\Omega)}\|\psi^0\|_{L^4(\Omega)}  \|\tilde{\Bu}\|_{L^4(\Omega)} \\ 
\leq&~\frac{\min_i D_i}{4} \|\nabla  \tilde{\psi}\|_{L^2(\Omega)}^2 +  C \|\psi^0\|_{L^4(\Omega)}^2  \|\tilde{\Bu}\|_{L^4(\Omega)}^2\\
\leq&~\frac{\min_i D_i}{4} \|\nabla  \tilde{\psi}\|_{L^2(\Omega)}^2 +  C \|\nabla\psi^0\|_{L^2(\Omega)}^2   (\| \tilde{\psi}\|_{L^4(\Omega)}^2+ M_6  \varepsilon^\frac14  )\\
\leq&~\frac{\min_i D_i}{4} \|\nabla  \tilde{\psi}\|_{L^2(\Omega)}^2 +  C M_6 (\| \tilde{\psi}\|_{L^4(\Omega)}^2+ M_6  \varepsilon^\frac14  )\\
\leq&~\frac{\min_i D_i}{4} \|\nabla  \tilde{\psi}\|_{L^2(\Omega)}^2 +  C M_6 \| \tilde{\psi}\|_{L^2(\Omega)}\| \nabla\tilde{\psi}\|_{L^2(\Omega)} + CM_6^2  \varepsilon^\frac14  \\
\leq&~\frac{\min_i D_i}{2} \|\nabla  \tilde{\psi}\|_{L^2(\Omega)}^2 +  C M_6^2 \| \tilde{\psi}\|_{L^2(\Omega)}^2 + CM_6^2  \varepsilon^\frac14,
\end{aligned}
\end{equation}
\begin{equation}\label{convergence-5}
\begin{aligned}
|( \phi_b'\tilde{\Bu} \cdot \Be_z, \tilde{\psi})| =&~ \left|\int_\Omega   \phi_b' \tilde{\Bu} \cdot \Be_z  \tilde{\psi}\dd \Bx \right| \leq~ c_\Delta\delta^{-1} \|\tilde{\Bu}\|_{L^2(\Omega)}\|\tilde{\psi}\|_{L^2(\Omega)}\\
\leq&~C(\| \tilde{\psi}\|_{L^2(\Omega)}+ M_6^\frac12  \varepsilon^\frac14)\| \tilde{\psi}\|_{L^2(\Omega)}\\
\leq&~C \| \tilde{\psi}\|_{L^2(\Omega)}^2+ CM_6  \varepsilon^\frac12
\end{aligned}
\end{equation}
and
\begin{equation}\label{convergence-6}
\begin{aligned}
|(\tilde{D}\phi_b'', \tilde{\psi})| \leq& ~ 2c_\Delta\delta^{-2} \|\tilde{D}\|_{L^2(\Omega)}\|\tilde{\psi}\|_{L^2(\Omega)} \leq  C\varepsilon^\frac12\|\tilde{\psi}\|_{L^2(\Omega)}\\
\leq&~C \| \tilde{\psi}\|_{L^2(\Omega)}^2+ C\varepsilon.
\end{aligned}
\end{equation}

Finally, by virtue of Lemma \ref{W-embedding} and \ref{lemmaA1}, one has
\begin{equation}\label{convergence-7}
\begin{aligned}
|( \tilde{D} \nabla \psi^0, \nabla\tilde{\psi} )| 
\leq&~\frac{\min_iD_i}{4} \|\nabla  \tilde{\psi}\|_{L^2(\Omega)}^2 + C \|\tilde{D}\|_{L^4(\Omega)}^2\|\nabla\psi^0\|_{L^4(\Omega)}^2\\
\leq&~\frac{\min_iD_i}{4} \|\nabla  \tilde{\psi}\|_{L^2(\Omega)}^2 + C \varepsilon^\frac{1}{2}  \|\psi^0\|_{W}^2\\
\leq&~\frac{\min_iD_i}{4} \|\nabla  \tilde{\psi}\|_{L^2(\Omega)}^2 + C \varepsilon^\frac{1}{2}  \|\mathcal{L}\psi^0\|_{L^2(\Omega)}^2.
\end{aligned}
\end{equation}
Substituting \eqref{convergence-4}-\eqref{convergence-7} to \eqref{convergence-8} gives 
\begin{equation} \nonumber
\begin{aligned}
&~\frac{1}{2}\frac{\dd }{\dd t} \|\tilde{\psi}\|_{L^2(\Omega)}^2 + \frac{\min_i D_i}{4}\|\nabla \tilde{\psi}\|_{L^2(\Omega)}^2 \\
\leq&~ CM_6^2\|\tilde{\psi}\|_{L^2(\Omega)}^2+ C\varepsilon+ CM_6^2\varepsilon^\frac14+C\varepsilon^\frac12\|\mathcal{L}\psi^0\|_{L^2(\Omega)}^2 .
\end{aligned}
\end{equation}
Since $\tilde{\psi}(0)=0$, one utilizes  Gr\"onwall's inequality and Proposition \ref{H^1} to obtain
\begin{equation}\label{M8}
\begin{aligned}
\|\tilde{\psi}(t)\|_{L^2(\Omega)}^2 \leq&~ \int_0^t (C\varepsilon+ CM_6^2\varepsilon^\frac14+ C\|\mathcal{L}\psi^0(\tau )\|_{L^2(\Omega)}^2 \varepsilon^\frac12) e^{\int_\tau ^t CM_6^2\dd s} \dd \tau\\
\leq&~  e^{CM_6^2t} \left(Ct\varepsilon +  CM_6^2 t \varepsilon^\frac14+ C\varepsilon^\frac12\int_0^t  \|\mathcal{L}\psi^0(\tau )\|_{L^2(\Omega)}^2  \dd \tau\right)\\
\leq&~e^{CM_6^2t} \left( Ct  + CM_6^2 t  + CM_7 \right)\varepsilon^\frac14.
\end{aligned}
\end{equation}
This, together with \eqref{convergence-3-1} and \eqref{convergence-3}, gives 
\begin{equation}
\begin{aligned}
&~\|\tilde{\Bu}\|_{L^2(\Omega)}^2+ \| \tilde{\psi}\|_{L^2(\Omega)}^2+ \| \nabla\tilde{p}\|_{L^2(\Omega)}^2\\
\leq &~ C (\varepsilon^\frac14 e^{CM_6^2t} \left( Ct  + CM_6^2 t  + CM_7\right) + \varepsilon^\frac14  M_6 )\\
=&:\varepsilon^\frac14 M_8.
\end{aligned}
\end{equation}
Then we finish the proof of this theorem.
\end{proof}

\subsection{Convergence of attractors}
Now we investigate the long-time behavior of the sharp interface model with one thin layer. The proof follows a standard argument, combining the attracting property of the global attractor of the limit model with the uniform convergence (with respect to initial data taken from the global attractor) of trajectories on finite time intervals, as established in the preceding subsection.

We first define the semigroup $S_\varepsilon(t)$ associated with problem \eqref{Sharps}-\eqref{bc}. For any $\varepsilon>0$ and $t$,  set 
\[S_\varepsilon (t):\psi_0\mapsto \psi^\varepsilon(t).\]
By Proposition \ref{prop1}, $S_\varepsilon(t)$ is a continuous operator from $H$ to itself, and is continuous in $t$. Consequently, we arrive at the following theorem on the existence and convergence of the attractors corresponding to $\{S_\varepsilon(t)\}$.

\begin{theorem}\label{thm1}
  For any $\varepsilon\ge 0$, there exists a global attractor $\mathcal{A}_\varepsilon \subset H$ for the semigroup $\{S_\varepsilon(t)\}$ associated with the problem \eqref{Sharps}-\eqref{bc}, which is uniformly bounded in $V$ in $\varepsilon$, compact and connected in $H$. Furthermore, as $\varepsilon$ goes to $0^+$, the attractor $\mathcal{A}_\varepsilon$ converges to $\mathcal{A}_0$ in the following sense,
\begin{equation}\nonumber
d(\mathcal{A}_{\varepsilon}, \mathcal{A}_0) \to  0,
\end{equation}
where $d(\mathcal{A}_{\varepsilon}, \mathcal{A}_0)$ is the Hausdorff semi-distance between $\mathcal{A}_{\varepsilon}, \mathcal{A}_0$ defined by 
\begin{equation}\nonumber
d(\mathcal{A}_{\varepsilon}, \mathcal{A}_0) := \sup_{f\in \mathcal{A}_{\varepsilon}} \operatorname{dist} (f,\mathcal{A}_0)= \sup_{f\in \mathcal{A}_{\varepsilon}} \inf_{g\in \mathcal{A}_0}\|f-g\|_{L^2(\Omega)}.
\end{equation}
\end{theorem}
\begin{proof}
Proposition \ref{L^2}, together with Proposition \ref{H^1}, implies that 
\begin{equation}\nonumber
  B_1 =\left\{\psi\in H:~ \|\psi\|_{L^2(\Omega)}^2 < \frac{M_1H^2}{\min_i D_i}+1 ,~\|\sqrt{D} \nabla\psi\|_{L^2(\Omega)}^2 \leq M_5\right\}
\end{equation}
is an absorbing set for the semigroup $\{S_\varepsilon(t)\}$ for any $\varepsilon\ge 0$. More precisely, if $\psi_0^\varepsilon \in B(0, R) \subset H$, then $\psi^\varepsilon(t) = S_\varepsilon(t) \psi_0^\varepsilon$ enters the absorbing set $B_1$ for some $t \leq T_1(R) + 1$ and stays in it for all $t \geq T_1(R) + 1$, where $T_1(R) = \frac{2H^2}{\min_i D_i}\ln R$. Furthermore, the absorbing set $B_1$ is bounded in $V$ and hence, relatively compact in $H$. This ensures the uniformly compactness of $S_\varepsilon(t)$ for large $t$. Then, according to \cite[Theorem 1.1]{Temam1997}, the semigroup $\{S_\varepsilon(t)\}$ associated with problem \eqref{Sharps}-\eqref{bc} possesses a global attractor $\mathcal{A}_\varepsilon$, which is defined as the $\omega$-limits of $B_1$: 
\begin{equation}\nonumber
\mathcal{A}_\varepsilon = \bigcap_{s \geq 0} \overline{\bigcup_{t \geq s} S_\varepsilon(t) B_1}.
\end{equation}
In particular, $\mathcal{A}_\varepsilon$ is  compact and connected in $H$, as well as bounded in $V$.

To show the convergence, we assume for contradiction that the assertion is false. Then there exists an constant $c >0$ and a sequence $\{\varepsilon_k\}$ converging to $0^+$ such that
\[ d(\mathcal{A}_{\varepsilon_k},\mathcal{A}_0) \geq c ,~~~~\text{ for all } k \geq 1.
\]
Since the global attractor $\mathcal{A}_{\varepsilon_k}$ are compact in $H$, there exists an $a_k \in \mathcal{A}_{\varepsilon_k}$ such that
\[\operatorname{dist} (a_k,\mathcal{A}_0) =d(\mathcal{A}_{\varepsilon_k},\mathcal{A}_0) \geq c.\]
On the other hand, since $\mathcal{A}_0$ absorbs all the bounded sets in $H$, we choose $T > 0$ to be sufficiently large such that
\[d(S_0(T)B_1 ,\mathcal{A}_0)<\frac{c}{4}.\]
Since the set $\mathcal{A}_\varepsilon$ is invariant with respect to $S_\varepsilon(t)$ for any $t\ge 0$, i.e., $S_\varepsilon (t) \mathcal{A}_\varepsilon= \mathcal{A}_\varepsilon$, then for each $a_k \in \mathcal{A}_{\varepsilon_k}$, there exists $b_k \in \mathcal{A}_{\varepsilon_k}$ such that $S_{\varepsilon_k}(T)b_k= a_k$ for all $k$. This implies that $b_k\in \mathcal{A}_{\varepsilon_k} \subset B_1$ for all $k$. By Proposition \ref{finite-time}, we have
\begin{equation}\nonumber
\begin{aligned}
  \|a_k -S_0(T)b_k\|_{L^2(\Omega)} \leq&~ \|S_{\varepsilon_k}(T)b_k-S_0(T)b_k\|_{L^2(\Omega)}\\
  \leq&~ \varepsilon_k^\frac14 M_8\left(T,\sqrt{ \frac{M_1H^2}{\min_i D_i}+1},\sqrt{M_5}\right)\\
  \leq&~\frac{c}{4},
\end{aligned}
\end{equation}
provided $k$ is sufficiently large. Hence,
\begin{equation}\nonumber
\begin{aligned}
  \operatorname{dist}(a_k, \mathcal{A}_0) \leq &~\|a_k -S_0(T)b_k\|_{L^2(\Omega)} + \operatorname{dist}(S_0(T)b_k, \mathcal{A}_0)\\
  \leq&~  \|a_k -S_0(T)b_k\|_{L^2(\Omega)} + d(S_0(T)B_1, \mathcal{A}_0) \\
  \leq&~ \frac{c}{2}.
\end{aligned}
\end{equation}
This contradiction completes the proof of this theorem.
\end{proof}

\section{Conclusion}

In this work we have shown that when the thickness of a layer in a multilayer porous medium becomes sufficiently small, the thin layer can be neglected in convection problems. More precisely, we established convergence in the $L^{2}$ sense over arbitrary finite time intervals, provided the initial data lies in $H^{1}$, and we further demonstrated convergence of the global attractors with $L^{2}$ as the phase space. The case of a curved thin layer can be handled analogously by employing curvilinear coordinates, under the assumption that the layers remain well separated except at the thin layer under consideration, and that the thinness is sufficiently uniform to justify the coordinate transformation. The case of smooth variable tensor valued diffusivity and permeability can be handled similarly provided they are sufficiently smooth and uniformly elliptic.

A natural question that arises is whether convergence holds in the stronger $H^{1}$ norm. This is particularly relevant since predictions in applications are sometimes based on $H^{1}$ estimates of the simplified model. While weak convergence in $H^{1}$ follows directly, establishing strong convergence appears more delicate due to the difference in the domains of the principal linear operators in the full and limit problems, even when the initial data is in $H^{1}$. In addition, the permeability and diffusivity may be nonlinear functions of the state variable $\phi$. Proving the convergence in this case would be a challenge in general.

Our analysis has focused on the vanishing-thickness limit of a single layer. An even more intriguing and challenging direction is the simultaneous limit in which both the thickness and permeability of a layer vanish, a possible scenario in several geophysical applications. The resulting reduced model is expected, at least heuristically, to exhibit a Robin-type interfacial condition reminiscent of the imperfect bounding interface conditions studied in composite materials~\cite{AV2018}, but with the added complexity of nonlinear coupling.  

These questions, along with related extensions, remain open and deserve further study.  

\bibliographystyle{amsalpha}
\bibliography{refs}

\end{document}